\documentclass[11pt]{amsart}
\usepackage{amsmath}
 
\usepackage{amsfonts} 
 \usepackage{amsthm}
 \usepackage{latexsym}
 \usepackage{amscd}
 \usepackage[all]{xy}
 \usepackage{mathrsfs}
 \usepackage{yhmath}
 \usepackage{appendix}
\usepackage{amssymb}
\usepackage{mathtools}
\def\RR{\mathbb{R}}

\def\CC{\mathbb{C}}

\def\odd{{\text{odd}}}

\DeclareMathOperator{\ind}{ind}

\DeclareMathOperator{\Fix}{Fix}
\DeclareMathOperator{\Tr}{Tr}
\DeclareMathOperator{\Per}{Per}

\theoremstyle{plain}
\newtheorem{thm}{Theorem}[section]
\newtheorem{prop}[thm]{Proposition}
\newtheorem{lemma}[thm]{Lemma}
\newtheorem{cor}[thm]{Corollary}

\theoremstyle{definition}

\newtheorem{cla}[thm]{Claim}

\newtheorem{rmk}[thm]{Remark}

\def\odd{{\text{odd}}}

 \usepackage{scalerel,stackengine}
\stackMath
\newcommand\reallywidecheck[1]{%
\savestack{\tmpbox}{\stretchto{%
  \scaleto{%
    \scalerel*[\widthof{\ensuremath{#1}}]{\kern-.6pt\bigwedge\kern-.6pt}%
    {\rule[-\textheight/2]{1ex}{\textheight}}
  }{\textheight}%
}{0.5ex}}%
\stackon[1pt]{#1}{\scalebox{-1}{\tmpbox}}%
}
\begin{document}

\pagestyle{plain}
\thispagestyle{plain}

\title[]{Positive hyperbolic periodic orbits for area-preserving maps and Reeb flows}
\author[Masayuki ASAOKA]{Masayuki ASAOKA}\address[Masayuki Asaoka]{Faculty of Science and Engineering, Doshisha University,
 1-3 Tatara Miyakodani, Kyotanabe 610-0394, JAPAN.}
\email{masaoka@mail.doshisha.ac.jp}

\author[Taisuke SHIBATA]{Taisuke SHIBATA}
\address[Taisuke Shibata]{Research Institute for Mathematical Sciences, Kyoto University, Kyoto 606-8502,
JAPAN.}
\email{shibata@kurims.kyoto-u.ac.jp}

\date{\today}

\begin{abstract} We prove that a strongly nondegenerate Reeb vector field on a closed contact three-manifold has infinitely many simple positive hyperbolic periodic orbits if it has at least three simple periodic orbits. The main ingredient is a result for area-preserving maps on compact surfaces, possibly with boundary. We prove that, under a natural boundary index condition, an area-preserving nondegenerate map with infinitely many periodic points has infinitely many positive hyperbolic periodic points in the interior. The proof combines this surface result with the broken book decomposition theorem of Colin--Dehornoy--Rechtman. \end{abstract}

\maketitle

\section{Introduction and main results}

The Weinstein conjecture states that every Reeb vector field on a
closed contact manifold has a periodic orbit. In dimension three, this
conjecture was proved by Taubes using Seiberg--Witten theory~\cite{Tau}.
After this existence result, it is natural to study the number and the
types of periodic orbits of Reeb flows. Concerning the number of
periodic orbits, Colin--Dehornoy--Rechtman proved that a nondegenerate
Reeb vector field on a closed three-manifold has either exactly two or
infinitely many simple periodic orbits
\cite[Theorem~1.2]{CDR}. The case of exactly two simple periodic orbits
is very special: the manifold is a lens space, including $S^3$, and
both simple periodic orbits are elliptic; see \cite{HT,CHP}. Therefore,
a nondegenerate Reeb vector field with at least three simple periodic
orbits has infinitely many simple periodic orbits.

We next consider the types of periodic orbits. A nondegenerate periodic
orbit of a Reeb flow on a three-manifold is elliptic, positive
hyperbolic, or negative hyperbolic. Cristofaro-Gardiner, Hutchings and
Pomerleano proved the existence of a simple positive hyperbolic
periodic orbit under some assumptions. They also asked whether every
nondegenerate Reeb vector field has a simple positive hyperbolic
periodic orbit unless it has exactly two simple elliptic periodic
orbits on a lens space~\cite{CHP}. In this paper, we consider the
following stronger question: if a nondegenerate Reeb vector field has
infinitely many simple periodic orbits, must it have infinitely many
simple positive hyperbolic periodic orbits?

Our main theorem gives an affirmative answer under the strong
nondegeneracy assumption. More precisely, we prove that a strongly
nondegenerate Reeb vector field on a closed contact three-manifold has
infinitely many simple positive hyperbolic periodic orbits, provided
that it has at least three simple periodic orbits. Here, strongly
nondegenerate means that all periodic orbits are nondegenerate and that
all intersections between stable and unstable manifolds of hyperbolic
periodic orbits are transverse.

The proof uses the broken book decomposition theorem of
Colin--Dehornoy--Rechtman~\cite{CDR}. If the broken book has a broken
binding component, the stable and unstable manifolds of the broken
binding components form a homoclinic or heteroclinic cycle. By strong
nondegeneracy, the intersections in this cycle are transverse, and
hence one obtains a Smale horseshoe. We then show that infinitely many
periodic orbits in this horseshoe can be chosen to be positive
hyperbolic.

If the broken book has no broken binding component, it is a rational
open book decomposition, and each page is a Birkhoff section. The
problem is then reduced to the first return map on a compact surface,
possibly with boundary. This leads to the main surface result of this
paper: under a natural boundary index condition, an area-preserving
nondegenerate map with infinitely many periodic points has infinitely
many positive hyperbolic periodic points in the interior. We also note
that Contreras and Mazzucchelli proved that every strongly nondegenerate
Reeb flow on a closed contact three-manifold admits a Birkhoff
section~\cite{CM}.

This surface result is related to classical results in area-preserving
surface dynamics. Franks proved several periodic point theorems for
area-preserving homeomorphisms of the annulus and of surfaces of genus
zero~\cite{Fr1,Fr2,Fr3}. Such results can be applied to Reeb flows
through Birkhoff sections. For example, Hofer, Wysocki and Zehnder
constructed disk-like Birkhoff sections for convex Reeb flows on $S^3$
and studied their periodic orbits using the first return map
\cite{HWZ2}. In our proof, results of Franks and Le Calvez are used to
produce periodic points of large prime periods. We then combine these
results with the Lefschetz fixed point formula and the fixed point
index to obtain infinitely many positive hyperbolic periodic points.

The paper is organized as follows. In the rest of this section, we
state the main theorem for Reeb flows and the main result for
area-preserving maps. In Section~2, we prove the result on surfaces. In
Section~3, we combine this result with the broken book decomposition
theorem to prove the main theorem for Reeb flows.

\subsection{Main result for Reeb flows} Let $(M,\lambda)$ be a closed contact three-manifold. The Reeb vector field $R_\lambda$ is defined by \[ \lambda(R_\lambda)=1, \qquad d\lambda(R_\lambda,\cdot)=0. \] A periodic orbit is called simple if it is not an iterate of another periodic orbit. A periodic orbit $\gamma$ is called nondegenerate if the linearized return map along $\gamma$, restricted to the contact plane, has no eigenvalue equal to $1$. We say that $R_\lambda$ is nondegenerate if all its periodic orbits, including all iterates of simple periodic orbits, are nondegenerate. A nondegenerate periodic orbit is called positive hyperbolic if the eigenvalues of its linearized return map are positive real numbers. It is called negative hyperbolic if the eigenvalues are negative real numbers, and elliptic if the eigenvalues lie on the unit circle of $\CC$. We say that $R_\lambda$ is strongly nondegenerate if it is nondegenerate and all intersections between the stable and unstable manifolds of hyperbolic periodic orbits are transverse (see \cite{CDR}). This condition is $C^\infty$-generic among contact forms; it follows from the Kupka--Smale transversality argument, see \cite{Rob} and the discussion in \cite[p593]{CM}. Our main theorem is the following. 

\begin{thm}\label{thm:strongly} Let $(M,\lambda)$ be a closed contact three-manifold. Suppose that $R_\lambda$ is strongly nondegenerate. If $R_\lambda$ has at least three simple periodic orbits, then it has infinitely many simple positive hyperbolic periodic orbits. 
\end{thm}

Irie proved that, for a $C^\infty$-generic contact form defining a fixed contact structure on a closed three-manifold, the union of periodic Reeb orbits is dense~\cite{Iri}. In particular, such a Reeb flow has infinitely many simple periodic orbits. Together with the genericity of the strongly nondegenerate condition, Theorem~\ref{thm:strongly} gives the following corollary. 

\begin{cor}\label{cor:generic-positive-hyp}
Let $(M,\xi)$ be a closed contact three-manifold. For a $C^\infty$-generic contact form $\lambda$ with $\ker\lambda=\xi$, the Reeb vector field $R_\lambda$ has infinitely many simple positive hyperbolic periodic orbits. \end{cor}

\subsection{Main result for area-preserving maps} 

We next state the result on surfaces used in the proof of Theorem~\ref{thm:strongly}. Let $\Sigma$ be a surface and let \[ f:\Sigma\to\Sigma \] be a diffeomorphism. A point $p\in\Sigma$ is called a periodic point of minimal period $n\in\mathbb{Z}_{>0}$ if \[ f^n(p)=p \] and \[ f^m(p)\neq p \] for every $0<m<n$. We say that $f$ is nondegenerate if, for every $n\geq1$ and every $p\in\Fix(f^n)$, the linear map \[ df^n(p):T_p\Sigma\to T_p\Sigma \] has no eigenvalue equal to $1$. Suppose that $f$ preserves an area form. A periodic point $p$ of minimal period $n$ is called positive hyperbolic if the eigenvalues of $df^n(p)$ are positive real numbers, and negative hyperbolic if they are negative real numbers. It is called elliptic if the eigenvalues lie on the unit circle. Every nondegenerate periodic point of an area-preserving map is of exactly one of these three types. 

For an isolated fixed point $p$, we denote by $\ind(p,f)$ its fixed point index. If $p\in\partial\Sigma$, we attach an outer collar to $\Sigma$ and extend $f$ over the added collar so that points in the collar are pushed toward $\partial\Sigma$. We define $\ind(p,f)$ as the usual local fixed point index of this extension. With this convention, the Lefschetz fixed point formula takes the usual form. For a nondegenerate fixed point in the interior, the fixed point index is $+1$ for elliptic and negative hyperbolic points, and $-1$ for positive hyperbolic points. The result is as follows.

\begin{thm}\label{maintheoremonsurface}
Let $\Sigma$ be a compact surface, possibly with boundary, and let $\omega$ be an area form on $\mathrm{Int}\,\Sigma$ such that \[ \int_{\mathrm{Int}\,\Sigma}\omega<\infty. \] Let $f:\Sigma\to\Sigma $ be a homeomorphism such that $f|_{\mathrm{Int}\,\Sigma}: \mathrm{Int}\,\Sigma\to\mathrm{Int}\,\Sigma $ is an $\omega$-preserving diffeomorphism. Assume that: 
\begin{itemize} 
\item[(1)] $f|_{\mathrm{Int}\,\Sigma}$ is nondegenerate and has infinitely many periodic points in $\mathrm{Int}\,\Sigma$; \item[(2)] $f$ has only finitely many periodic points on $\partial\Sigma$, and for every $n\in\mathbb{Z}_{>0}$ and every \[ p\in\Fix(f^n)\cap\partial\Sigma, \] one has \[ \ind(p,f^n)\in\{-1,0\}. \] 
\end{itemize} 
Then $f$ has infinitely many positive hyperbolic periodic points in $\mathrm{Int}\,\Sigma$. 

\end{thm} 

As an immediate consequence, we obtain the closed surface case. \begin{cor}\label{cor:closed-surface} 
Let $\Sigma$ be a closed surface and let $ f:\Sigma\to\Sigma $ be an area-preserving diffeomorphism. Suppose that $f$ is nondegenerate and has infinitely many periodic points. Then $f$ has infinitely many positive hyperbolic periodic points. 
\end{cor}

\subsection*{Acknowledgments}
The authors would like to thank the anonymous referee for many helpful comments and suggestions. Thanks to these comments, the original paper was substantially revised and greatly improved.
TS would like to thank his advisor, Professor Kaoru Ono, for his support. MA was supported by JSPS KAKENHI Grant Number 22K03302. TS was supported by JSPS KAKENHI Grant Number JP21J20300.
\section{Proof of Theorem~\ref{maintheoremonsurface}}

In this section, we prove Theorem~\ref{maintheoremonsurface}. The proof is divided into two parts. First, we prove that, after taking a suitable iterate, one can find periodic points of all sufficiently large prime periods. This part uses results due to Franks and Le Calvez. Second, we use the Lefschetz fixed point formula and the fixed point index to show that infinitely many periodic points must be positive hyperbolic.

\subsection{Periodic points of large prime period}

We start with the case of closed surfaces of positive genus. The next proposition is an important ingredient, and its proof is based on the strategy of Le Calvez~\cite{LeC}.

\begin{prop}\label{nonwandering}
Let $\Sigma$ be a closed surface of positive genus, and let
$f \colon \Sigma \to \Sigma$ be a non-wandering homeomorphism.
Suppose that $f$ has zero topological entropy and has periodic points of
arbitrarily large period.
Then there exists $n \in \mathbb{Z}_{>0}$ with the following property:
there exists $M \in \mathbb{Z}_{>0}$ such that for any prime number
$q > M$, the map $f^n$ has a periodic point of period $q$.
\end{prop}

\begin{proof}[\bf Proof of Proposition~\ref{nonwandering}]
If necessary, replacing $f$ by $f^2$, we may assume that $f$ is
orientation preserving.

We use the arguments of Le Calvez~\cite{LeC} and only give the outline.
We distinguish the cases where $\Sigma$ has genus at least $2$ and where
$\Sigma$ has genus $1$.

\medskip

\noindent\textbf{Case 1: $\mathrm{genus}(\Sigma)\ge2$.}

By the Nielsen--Thurston decomposition theorem, after passing to an
iterate, there exists $N\in\mathbb{Z}_{>0}$ such that
\[
g:=f^N
\]
is isotopic either to a Dehn twist map or to the identity.

First, assume that $g$ is isotopic to a Dehn twist map. We use the notation
of \cite[Section~3]{LeC}. Thus Le Calvez considers an annular
covering $\widehat S$ and the induced map $\widehat g$ on $\widehat S$.
By \cite[Proposition~3.1]{LeC}, at least one of the following two
situations occurs:
\begin{itemize}
\item[(1)] $g$ has periodic points of arbitrarily large period;
\item[(2)] there exists an essential simple loop
$\widehat\lambda\subset\widehat S$ such that
\[
\widehat g(\widehat\lambda)\cap\widehat\lambda=\emptyset .
\]
\end{itemize}

Suppose first that (2) does not occur. As discussed in
\cite[Proposition~3.1]{LeC}, for every
rational number $p/q$ in a certain nontrivial interval, there exists a periodic point of period $q$. Hence, by taking $q$ to
be a sufficiently large prime number, one obtains periodic points of period $q$ for every sufficiently large prime number $q$.

Suppose next that (2) occurs. Then \cite[Proposition~3.2]{LeC}
and the proof of \cite[Proposition~3.6]{LeC} give periodic points
of every sufficiently large period. We use the notation of
\cite[Proposition~3.6]{LeC}: let $\widetilde S$ be the universal
covering of $\widehat S$, and let
\[
\widetilde g:\widetilde S\to\widetilde S
\]
be a lift of $\widehat g$. Then there exist a covering automorphism $T_1$
and an integer $m_5$ such that, for every $m\ge m_5$, the map
\[
\widetilde g^{\,m}\circ T_1^{-1}
\]
has a fixed point. Hence $g^m$ has a fixed point $z_m$. Moreover, the
period of $z_m$ for $g$ goes to $+\infty$ as $m\to+\infty$. Therefore, if
$q$ is a sufficiently large prime number, then $z_q$ is not fixed for $g$,
and its period for $g$ must be exactly $q$. Thus $g$ has a periodic point
of period $q$ for every sufficiently large prime number $q$.

Next, assume that $g$ is isotopic to the identity. This case is treated in
\cite[Section~4 and Proposition~1.5]{LeC}. In the proof of
\cite[Proposition~1.5]{LeC}, Le Calvez again reduces, after
passing to a suitable annular covering, to an annular situation. If the
Poincare--Birkhoff type argument applies, one obtains periodic points of
period $q$ for every sufficiently large prime number $q$. In the remaining
case, the same type of forcing argument as in the Dehn twist case gives
fixed points of maps of the form
\[
\widetilde g^{\,m}\circ T^{-1}
\]
for all sufficiently large $m$. Hence, by taking $m=q$ prime and sufficiently
large, one obtains periodic points of $g$ of period $q$.

Thus, in the case $\mathrm{genus}(\Sigma)\ge2$, there exist $N>0$ and
$M>0$ such that, for every prime number $q>M$, the map $f^N$ has a
periodic point of period $q$.

\medskip

\noindent\textbf{Case 2: $\Sigma=T^2$.}

We use the notation and the results of \cite[Section~5]{LeC}.
Since the topological entropy of $f$ is zero, after passing to an iterate,
there exists $N\in\mathbb{Z}_{>0}$ such that
\[
g:=f^N
\]
is isotopic either to a Dehn twist map or to the identity.

First, assume that $g$ is isotopic to a Dehn twist map. Consider a lift
\[
\widehat g:\mathbb{T}\times\mathbb{R}\to \mathbb{T}\times\mathbb{R}
\]
and its vertical rotation set $\operatorname{vrot}(\widehat g)$, as in
\cite[Section~5]{LeC}. Note that $\operatorname{vrot}(\widehat g)$ is  a non empty segment of $\RR$. In addition,  $\operatorname{vrot}(\widehat g)$ is reduced to an irrational number point when $f$ has no periodic point, otherwise $f$ has  periodic points of period arbitrarily large \cite[Proposition~5.1]{LeC}. Suppose that
$\operatorname{vrot}(\widehat g)$ is not reduced to a point. Then, for every
rational number
\[
\frac{p}{q}\in \operatorname{int}(\operatorname{vrot}(\widehat g)),
\]
there exists $\widehat z\in\mathbb{T}\times\mathbb{R}$ such that
\[
\widehat g^{\,q}(\widehat z)=V^p(\widehat z),
\]
where $V$ is the vertical deck transformation. If $p$ and $q$ are relatively
prime, then $\widehat z$ projects to a periodic point of $g$ of period $q$.
Fix a compact interval
\[
J\subset \operatorname{int}(\operatorname{vrot}(\widehat g)).
\]
Then there exists $M>0$ such that, for every prime number $q>M$, we can
choose an integer $p$ with $p/q\in J$ and $p\not\equiv 0 \pmod q$. Since
$q$ is prime, we have $\gcd(p,q)=1$. Hence $g$ has a periodic point of
period $q$ for every sufficiently large prime number $q$.

Suppose next that $\operatorname{vrot}(\widehat g)$ is reduced to a rational
number $p/q$. Then, as in the proof of
\cite[Proposition~5.1]{LeC}, one replaces $g$ by $g^q$ and
$\widehat g$ by
\[
V^{-p}\circ \widehat g^{\,q},
\]
so that the new vertical rotation set is $\{0\}$. The forcing argument in the
proof of \cite[Proposition~5.1]{LeC}, together with the
non-wandering assumption, then gives periodic points of sufficiently large
prime periods after passing to an iterate.

Next, assume that $g$ is isotopic to the identity. This case is treated in
\cite[Proposition~5.2]{LeC}. If the rotation set contains rational
vectors in its interior, the realization theorem for rational rotation vectors
gives periodic points of period $q$ for every sufficiently large prime number
$q$. In the remaining case where the proof reduces to an annular situation,
the same forcing argument gives periodic points with sufficiently large prime
periods after passing to an iterate.

Therefore, also in the case $\Sigma=T^2$, there exist $N'>0$ and $M>0$
such that, for every prime number $q>M$, the map
\[
g^{N'}=f^{NN'}
\]
has a periodic point of period $q$.

Combining the two cases, we obtain $n\in\mathbb{Z}_{>0}$ and
$M\in\mathbb{Z}_{>0}$ such that, for every prime number $q>M$, the map
$f^n$ has a periodic point of period $q$.
\end{proof}

\begin{prop}\label{propgenus0}
Let $\Sigma$ be a compact surface of genus zero, possibly with boundary.
Let $\omega$ be an area form on $\mathrm{Int}\,\Sigma$ such that
$
\int_{\mathrm{Int}\,\Sigma}\omega<\infty .
$
Let $f \colon \Sigma \to \Sigma$ be a homeomorphism such that
$f|_{\mathrm{Int}\,\Sigma}$ is an $\omega$-preserving diffeomorphism of
$\mathrm{Int}\,\Sigma$. Assume that $f$ satisfies the assumptions of
Theorem~\ref{maintheoremonsurface}. Then there exists $n \in \mathbb{Z}_{>0}$ with the following property:
there exists $M \in \mathbb{Z}_{>0}$ such that for any prime number
$q > M$, the map $f^n$ has a periodic point of period $q$.
\end{prop}

To prove Proposition~\ref{propgenus0}, we recall a theorem of Franks.

Let $A=S^1\times[0,1]$ be the closed annulus, and let
\[
\pi:\mathbb{R}\times[0,1]\to A
\]
be the universal covering map. Let
\[
T(u,v)=(u+1,v)
\]
be the deck transformation. Let $h:A\to A$ be a homeomorphism isotopic to
the identity, and let $\tilde h$ be a lift of $h$.

For $\tilde x\in\mathbb{R}\times[0,1]$, if the limit exists, we define
\[
\tau(\tilde x)
:=
\lim_{n\to\infty}
\frac{(\tilde h^{\,n}(\tilde x))_1-(\tilde x)_1}{n}.
\]
Here $(\cdot)_1$ denotes the first coordinate. If $x\in A$ and $\tilde x$ is
a lift of $x$, then $\tau(\tilde x)$ modulo $\mathbb{Z}$ does not depend on
the choice of the lift. We denote it by
\[
\rho(x)\in \mathbb{R}/\mathbb{Z}.
\]

\begin{thm}\cite[Corollary~2.4]{Fr1}\cite[Theorem~2.1]{Fr2}\label{thm:Franks}
Let $h$ be a homeomorphism of the annulus $A=S^1\times[0,1]$ which is
isotopic to the identity. Assume that $h$ is non-wandering. Let $\tilde h$
be a lift of $h$ to $\mathbb{R}\times[0,1]$. Suppose that there exist
points $\tilde x,\tilde y\in\mathbb{R}\times[0,1]$ such that
$\tau(\tilde x)$ and $\tau(\tilde y)$ exist and
\[
\tau(\tilde x)<\tau(\tilde y).
\]
Then, for any coprime integers $m,n$ with $n\ge1$ and
\[
\tau(\tilde x)<\frac{m}{n}<\tau(\tilde y),
\]
there exists $\tilde z\in\mathbb{R}\times[0,1]$ such that
\[
\tilde h^{\,n}(\tilde z)=T^m(\tilde z).
\]
In particular, $\pi(\tilde z)$ is a periodic point of $h$ of period $n$.
\end{thm}
\begin{rmk}
In our situation, the annulus maps preserve a finite measure. Hence, by the Poincare recurrence theorem, they are
non-wandering.
\end{rmk}

\begin{cor}\label{cor:Franks}
Under the assumptions of Theorem~\ref{thm:Franks}, suppose that there exist
$x,y\in A$ such that $\rho(x)$ and $\rho(y)$ are defined and
\[
\rho(x)\neq \rho(y).
\]
Then $h$ has periodic points of period $q$ for every sufficiently large
prime number $q$.
\end{cor}

\begin{proof}[\bf Proof of Corollary~\ref{cor:Franks}]
Choose lifts $\tilde x,\tilde y$ so that
$
\tau(\tilde x)<\tau(\tilde y).
$
For every sufficiently large prime number $q$, we can choose an integer
$p$ such that
$
\tau(\tilde x)<\frac{p}{q}<\tau(\tilde y)
$
and $p\not\equiv 0 \pmod q$. 
Then Theorem~\ref{thm:Franks} gives a periodic point  of period $q$.
\end{proof}

\begin{proof}[\bf Proof of Proposition~\ref{propgenus0}]
Choose sufficiently many distinct periodic points in $\mathrm{Int}\,\Sigma$. After replacing $f$ by an iterate, we may assume that all these points are fixed and that every boundary component of $\Sigma$ is $f$-invariant. Since $\Sigma$ has genus zero, the induced map on $H_1(\Sigma;\mathbb{R})$ is the identity. 

Since the fixed point indices on the boundary are nonpositive,  if the number of the chosen periodic points is sufficiently large, the Lefschetz fixed point formula implies that one of the fixed points in $\mathrm{Int}\,\Sigma$ has index $1$. We denote such a point by $p_*$. We also choose another fixed point $ p_0\in\mathrm{Int}\,\Sigma,\qquad p_0\neq p_*. $
Since $f$ is nondegenerate, $p_*$ is either elliptic or negative hyperbolic. The same argument applied to $f^2$ shows that $f^2$ also has a fixed point of index $1$ in $\mathrm{Int}\,\Sigma$. We will use this fact only in the last case below.

Collapse each boundary component of $\Sigma$ to a point. The quotient space is homeomorphic to $S^2$, and $f$ induces a homeomorphism of this sphere. Notice that each point obtained by collapsing a boundary component is an isolated fixed point of the induced map. Indeed, otherwise a sequence of interior fixed points would accumulate at a boundary fixed point, contradicting the isolation implicit in the boundary fixed point index.

Blowing up the two fixed points $p_*$ and $p_0$, we obtain
a closed annulus $A$ and an induced homeomorphism
\[
g:A\to A
\]
which is isotopic to the identity.
 Since $f$ preserves a finite area form on $\mathrm{Int}\,\Sigma$, the
annulus map $g$ is non-wandering.

Since $\ind(p_*,f)=1$, the point $p_*$ is either elliptic or negative hyperbolic. 
Therefore the rotation number $\rho_{p_*}$ is irrational in the
elliptic case, and is equal to $\frac12$ in the negative hyperbolic case.

If $\rho_{p_{0}}\neq \rho_{p_*}$,
then Corollary~\ref{cor:Franks} applies to $g$. Hence $g$, and therefore $f$, has periodic points of period $q$ for every
sufficiently large prime number $q$.

Suppose next that
$
\rho_{p_0}=\rho_{p_*}
$
and that this  value is irrational.
Since $f$ has infinitely many periodic points in $\mathrm{Int}\,\Sigma$, we can choose a periodic point
$q_*$ which is contained in the interior of the annulus $A$.  
Then $q_*$ has rational rotation number for the induced annulus map, while the boundary rotation numbers are irrational. Hence Corollary~\ref{cor:Franks} implies
that a suitable iterate of $f$ has periodic points of period $q$ for everysufficiently large prime number $q$.

It remains to consider the case $\rho_{p_0}=\rho_{p_*}=\frac12. $ 
In this case, $p_*$ is negative hyperbolic for $f$, and hence positive hyperbolic for $f^2$.
In particular, $ \ind(p_*,f^2)=-1, $ and the rotation number obtained by blowing up $p_*$ for $f^2$ is $0$. By the observation above, there exists a fixed point $r_*\in\mathrm{Int}\,\Sigma $ of $f^2$ such that $ \ind(r_*,f^2)=1. $ Since $f$ is nondegenerate, $r_*$ is either elliptic or negative hyperbolic for $f^2$. 
Thus the rotation number obtained by blowing up $r_*$ is either irrational or equal to $\frac12$. Collapse each boundary component of $\Sigma$ to a point and blow up the two fixed points $p_*$ and $r_*$.
We obtain a closed annulus $A_2$ and an induced homeomorphism $ g_2:A_2\to A_2 $ which is isotopic to the identity. As before, $g_2$ is non-wandering. The rotation numbers on the two boundary components are distinct. 
Therefore, Corollary~\ref{cor:Franks} applies to $g_2$. Hence $f^2$ has periodic points of period $q$ for every sufficiently large prime number $q$.

In all cases, there exists $n\in\mathbb{Z}_{>0}$ such that $f^n$ has a periodic point of period $q$ for every sufficiently large prime number $q$. This proves Proposition~\ref{propgenus0}.
\end{proof}

\subsection{The Lefschetz index argument}

Let $\Sigma$ be a surface.
For a map $f$, we denote by $\Fix(f)$ the set of its fixed points, and by
\[
\Per^{\odd}(f) := \bigcup_{\substack{n \ge 1 \\ n \text{ odd}}} \Fix(f^n)
\]
the set of periodic points of $f$ with odd period.
Note that $\Per^{\odd}(f) \subset \Per^{\odd}(f^2)$.
Hence,
\begin{equation*}
    \Per^{\odd}(f) \subset \Per^{\odd}(f^2) \subset \cdots \subset \Per^{\odd}(f^{2^{n-1}}) \subset \Per^{\odd}(f^{2^n}) \subset \cdots .
\end{equation*}

\begin{lemma}\label{iterateodd}
Let $\Sigma$ be a closed surface with positive genus, and let
$f \colon \Sigma \to \Sigma$ be a non-wandering homeomorphism.
Suppose that $f$ has zero topological entropy and has periodic points of
arbitrarily large period.
Then there exists a nonnegative integer $m$ such that
$\#\Per^{\odd}(f^{2^m}) = \infty$.
\end{lemma}

\begin{proof}[\bf Proof of Lemma \ref{iterateodd}]
It follows from Proposition \ref{nonwandering} that there exists a positive integer $n$ such that $f^n$ has a periodic point of period $q$ for every sufficiently large prime $q$.
Write $n = 2^m k$, where $m$ is a nonnegative integer and $k$ is odd.
Then $f^{2^m}$ has periodic points of arbitrarily large odd period, and hence $\#\Per^{\odd}(f^{2^m}) = \infty$.
\end{proof}

Recall that the Lefschetz number of a map $f \colon \Sigma \to \Sigma$ is defined by
$$L(f) := \sum_{i=0}^2 (-1)^i
\Tr\!\left(f_* \big|_{H_i(\Sigma;\mathbb{Q})}\right).$$

\begin{prop}\label{prop1}
Let $\Sigma$ be a compact surface, possibly with boundary, and let $\omega$
be an area form defined on $\mathrm{Int}\,\Sigma$ such that
$
\int_{\mathrm{Int}\,\Sigma}\omega < \infty .
$
Let $f \colon \Sigma \to \Sigma$ be a homeomorphism such that
$f|_{\mathrm{Int}\,\Sigma}$ is an $\omega$-preserving diffeomorphism of
$\mathrm{Int}\,\Sigma$ satisfying the assumptions of
Theorem~\ref{maintheoremonsurface}.
Assume further that there exists $C>0$ such that
\[
|L(f^n)|<C
\]
for all $n\ge1$, and that there exists $m_0\ge0$ such that
\[
\#\Per^{\odd}(f^{2^{m_0}})=\infty .
\]
Then $f$ has infinitely many positive hyperbolic periodic points.
\end{prop}

We will prove Proposition~\ref{prop1} using several lemmas.
From now on and until the end of the proof of Proposition~\ref{prop1}, we assume that $\Sigma$ and $f$ satisfy the assumptions of Proposition~\ref{prop1}.

Let $\Per_{h+}(f)$ denote the set of positive hyperbolic periodic points of $f$.

\begin{lemma}\label{lemma:odd1}
If $\Per^{\odd}(f)$ is infinite, then $\Per^{\odd}(f)\cap \Per_{h+}(f)$ is infinite.
\end{lemma}

\begin{proof}[\bf Proof of Lemma \ref{lemma:odd1}]
Assume that $\Per^{\odd}(f)$ is infinite but $\Per^{\odd}(f)\cap \Per_{h+}(f)$ is finite.
Let $B$ be the (finite) set of periodic points of $f$ contained in $\partial\Sigma$, and set
\[
K:= \#\bigl((\Per^{\odd}(f)\cap \Per_{h+}(f))\cup B\bigr).
\]
For an odd integer $n\ge 1$, define
\[
\Lambda_n := \Fix(f^n)\cap
\bigl((\Per^{\odd}(f)\cap \Per_{h+}(f))\cup B\bigr).
\]
For each $p\in \Per_{h+}(f)$, we have $\ind(p,f^n)=-1$ whenever $p\in \Fix(f^n)$.
For each $p\in B$, the assumptions imply $\ind(p,f^{n})\in\{-1,0\}$.
Hence,
\begin{equation*}\label{eq:Lambda-bound}
\sum_{p\in \Lambda_n} \ind(p,f^{n}) \ge -K
\end{equation*}
for every odd $n$.

Since $\Per^{\odd}(f)$ is infinite while $\Per^{\odd}(f)\cap \Per_{h+}(f)$ is finite, we can choose pairwise distinct points
\[
p_1,\ldots,p_{K+M'+1}\in \Per^{\odd}(f)\setminus\bigl(\Per_{h+}(f)\cup B\bigr),
\]
where $M'$ is  an integer with $M'>C$.
Let $r_j$ denote the (minimal) period of $p_j$, and set $N:=\prod_{j=1}^{K+M'+1} r_j$.
Then $N$ is odd and $p_j\in \Fix(f^N)$ for each $j$.

By nondegeneracy and area preservation, every periodic point is either elliptic, positive hyperbolic, or negative hyperbolic;
moreover, for a nondegenerate fixed point of an iterate, the fixed point index is $+1$ for elliptic and negative hyperbolic points,
and $-1$ for positive hyperbolic points.
Since each $p_j$ is not positive hyperbolic, it follows that
\[
\ind(p_j,f^N)=1 \qquad \text{for } j=1,\ldots,K+M'+1.
\]
Therefore,
\begin{align*}
L(f^N)
&= \sum_{p\in \Fix(f^{N})} \ind(p, f^{N}) \\
&= \sum_{p\in \Fix( f^{N})\setminus \Lambda_N} \ind(p, f^{N})
 \;+\; \sum_{p\in \Lambda_N} \ind(p,f^{N}) \\
&\ge \sum_{j=1}^{K+M'+1} \ind(p_j,f^N) \;+\; \sum_{p\in \Lambda_N} \ind(p,f^{N}) \\
&\ge (K+M'+1) - K \\
&= M'+1>C.
\end{align*}
This contradicts the assumption $|L(f^n)|<C$ for all integers $n$.
Hence $\Per^{\odd}(f)\cap \Per_{h+}(f)$ must be infinite.
\end{proof}

\begin{lemma}\label{lemma:odd2}
Suppose that $\Per^{\odd}(f^{2^{m}})$ is finite and $\Per^{\odd}(f^{2^{m+1}})$ is infinite.
Then $\Per_{h+}(f)$ is infinite.
\end{lemma}

\begin{proof}[\bf Proof of Lemma \ref{lemma:odd2}]
Since $\Per^{\odd}(f^{2^{m}})$ is finite, the set
\[
\bigcup_{i=0}^{m}\Per^{\odd}(f^{2^{i}})
\]
is also finite.

By assumption, $\Per^{\odd}(f^{2^{m+1}})$ is infinite. Applying
Lemma~\ref{lemma:odd1} to the map $f^{2^{m+1}}$, we have that
\[
\Per^{\odd}(f^{2^{m+1}})\cap \Per_{h+}(f^{2^{m+1}})
\]
is infinite. Hence the subset
\[
S :=
\Per_{h+}(f^{2^{m+1}})\cap
\Bigl(\Per^{\odd}(f^{2^{m+1}})
\setminus \bigcup_{i=0}^{m}\Per^{\odd}(f^{2^{i}})\Bigr)
\]
is infinite as well.

Let $x\in S$, and let $r$ be the minimal period of $x$ under $f$.
Since $x\in \Per^{\odd}(f^{2^{m+1}})$, the minimal period of $x$ under
$f^{2^{m+1}}$ is odd. Hence the largest power of $2$ dividing $r$ is at
most $2^{m+1}$. On the other hand, since
\[
x\notin \bigcup_{i=0}^{m}\Per^{\odd}(f^{2^{i}}),
\]
the largest power of $2$ dividing $r$ is at least $2^{m+1}$. Therefore we
can write
\[
r=2^{m+1}q
\]
with $q$ odd. Since $x$ is positive hyperbolic for $f^{2^{m+1}}$, it is
positive hyperbolic for $f$.  Since $S$ is infinite, $\Per_{h+}(f)$ is also infinite.
\end{proof}

\begin{proof}[\bf Proof of Proposition~\ref{prop1}]
Choose $m$ to be the smallest nonnegative integer such that
\[
\#\Per^{\odd}(f^{2^m})=\infty .
\]

If $m=0$, it follows from Lemma~\ref{lemma:odd1} that
$\Per^{\odd}(f)$ contains infinitely many positive hyperbolic periodic
points. Hence $\Per_{h+}(f)$ is infinite.

If $m\ge1$, then $\Per^{\odd}(f^{2^{m-1}})$ is finite but
$\Per^{\odd}(f^{2^m})$ is infinite. Hence Lemma~\ref{lemma:odd2} implies
that $\Per_{h+}(f)$ is infinite.

In any case, $f$ has infinitely many positive hyperbolic periodic points.
\end{proof}

\subsection{Proof of Theorem~\ref{maintheoremonsurface}}
\begin{proof}[\bf Proof of Theorem~\ref{maintheoremonsurface}]
We first assume that $f$ has positive topological entropy. By Katok's
theorem~\cite{Kat}, there exist $r\geq1$ and a rectangle $R$ such that
$f^r$ has a horseshoe in $R$ and
\[
f^j(R)\cap R=\emptyset
\]
for every $1\leq j<r$. Thus, a point of minimal period $k$ for $f^r$ in this horseshoe has minimal period $rk$ for $f$. Choosing periodic words for which the action on the stable direction is orientation preserving, we obtain infinitely many periodic points with positive real eigenvalues. Hence $f$ has infinitely many positive hyperbolic periodic points.

Next, assume that the topological entropy of $f$ is zero. We need the next claim  to apply  Proposition~\ref{prop1}  to $f$. 

\begin{cla}\label{lefschetznumber}
Let $g \colon \Sigma \to \Sigma$ be a homeomorphism of a compact surface,
possibly with boundary. Suppose that $g$ has zero topological entropy.
Then there exists $C>0$ such that
\[
|L(g^n)|<C
\]
for all $n\ge1$.
\end{cla}

\begin{proof}[\bf Proof of Claim~\ref{lefschetznumber}] By the Nielsen--Thurston classification, since $g$ has zero topological entropy, there exist $p\in\mathbb{Z}_{>0}$ and pairwise disjoint simple closed curves \[ \gamma_1,\ldots,\gamma_\ell\subset\Sigma \] such that $g^p$ is isotopic to a product of Dehn twists along these curves. Hence all eigenvalues of \[ (g^p)_*:H_1(\Sigma;\mathbb{Q})\to H_1(\Sigma;\mathbb{Q}) \] are equal to $1$. Therefore, all eigenvalues of \[ g_*:H_1(\Sigma;\mathbb{Q})\to H_1(\Sigma;\mathbb{Q}) \] are roots of unity. It follows that \[ \Tr\!\left((g^n)_*\big|_{H_1(\Sigma;\mathbb{Q})}\right) \] is bounded independently of $n$. The contributions from $H_0(\Sigma;\mathbb{Q})$ and $H_2(\Sigma;\mathbb{Q})$ are also bounded. Therefore, $L(g^n)$ is bounded. \end{proof}

 By
Claim~\ref{lefschetznumber}, there exists $C>0$ such that
\[
|L(f^n)|<C
\]
for all $n\ge1$.

It remains to find infinitely many odd periodic points after taking an iterate. Suppose first that $\Sigma$ has positive genus. Collapse each boundary component of $\Sigma$ to a point and denote the resulting closed surface and induced map by $\widehat\Sigma$ and $\widehat f$, respectively. Then $\widehat f$ is non-wandering, has zero topological entropy, and has periodic points of arbitrarily large period. 
Hence Lemma~\ref{iterateodd} implies that $ \#\Per^{\odd}(\widehat f^{2^m})=\infty $ for some $m\geq0$. Since there are only finitely many collapsed boundary points, we obtain $ \#\Per^{\odd}(f^{2^m})=\infty. $ If $\Sigma$ has genus zero, the same conclusion follows from Proposition~\ref{propgenus0}.

Therefore Proposition~\ref{prop1} applies to $f$, and we conclude that $f$
has infinitely many positive hyperbolic periodic points. This completes the proof of Theorem~\ref{maintheoremonsurface}.
\end{proof}

\section{Proof of Theorem \ref{thm:strongly}}

By Colin--Dehornoy--Rechtman \cite[Theorem~1.1]{CDR}, the Reeb vector
field $R_\lambda$ is $\partial$-strongly carried by a broken book
decomposition. If the broken book decomposition has a broken binding
component, every broken binding component is a hyperbolic periodic
orbit \cite[Definition~2.6]{CDR}, and the stable and unstable manifolds
of the broken binding components give homoclinic or heteroclinic cycles
\cite[Lemma~4.11]{CDR}. Since $R_\lambda$ is strongly nondegenerate,
all these intersections are transverse. Hence, as explained in
\cite[Section~4.4]{CDR}, one obtains a Smale horseshoe. On the other
hand, if the broken book decomposition has no broken binding component,
then it is a rational open book decomposition, and each page is a
$\partial$-strong Birkhoff section for $R_\lambda$; see
\cite[Definitions~2.1, 2.2 and 2.7]{CDR}. Thus the proof reduces to
these two cases.

We first show two facts which will be used below. \begin{lemma}\label{lem:positive-hyp-from-horseshoe} Let $R_\lambda$ be a strongly nondegenerate Reeb vector field on a closed three-manifold. Suppose that $R_\lambda$ has a transverse homoclinic orbit to a hyperbolic periodic orbit, or more generally a transverse heteroclinic cycle between hyperbolic periodic orbits. Then $R_\lambda$ has infinitely many simple positive hyperbolic periodic orbits. \end{lemma} \begin{proof}[\bf Proof of Lemma~\ref{lem:positive-hyp-from-horseshoe}] By the Smale--Birkhoff homoclinic theorem, or by the standard construction near a transverse heteroclinic cycle, there exists a local Poincar\'e return map with a horseshoe. This return map preserves the area form induced by $d\lambda$. As in the positive entropy case in the proof of Theorem~\ref{maintheoremonsurface}, the horseshoe contains infinitely many positive hyperbolic periodic points of arbitrarily large minimal period. They correspond to distinct simple positive hyperbolic periodic orbits of the Reeb flow. \end{proof}

\begin{lemma}\label{lem:birkhoff-boundary-index}
Let $S$ be a $\partial$-strong Birkhoff section for a nondegenerate
Reeb vector field $R_\lambda$, and let
\[
\varphi:\operatorname{Int}S\to\operatorname{Int}S
\]
be the first return map. Then $\varphi$ extends to a homeomorphism of
$S$. Moreover, $\varphi$ has only finitely many periodic points on
$\partial S$, and for every $n\geq1$ and every
\[
p\in\operatorname{Fix}(\varphi^n)\cap\partial S,
\]
one has
\[
\operatorname{ind}(p,\varphi^n)\in\{-1,0\}.
\]
\end{lemma}

\begin{proof}[\bf Proof of Lemma~\ref{lem:birkhoff-boundary-index}]
By the $\partial$-strong condition, the boundary of the compactified
page is a transverse section for the linearized flow on the unit normal
torus of each boundary orbit. Hence the first return map and its inverse
extend continuously to the boundary, and $\varphi$ extends to a
homeomorphism of $S$; see \cite[Definition~2.2 and the proof of
Corollary~4.8]{CDR}.

Let $\gamma$ be a boundary orbit of $S$. If $\gamma$ is elliptic, then
the induced map on the corresponding boundary component is an
irrational rotation. Hence it has no periodic point. If $\gamma$ is
hyperbolic, the periodic points on the boundary correspond to the
stable and unstable directions of $\gamma$. Thus $\varphi$ has only
finitely many periodic points on $\partial S$.

Let
\[
p\in\operatorname{Fix}(\varphi^n)\cap\partial S
\]
and put $\Phi=\varphi^n$. We use the boundary index convention introduced
above: a half-disk neighborhood of $p$ is extended slightly across the
boundary, and the added outside collar is pushed toward the boundary.

Suppose first that $p$ corresponds to an unstable direction. The
projectivized return map contracts in the boundary direction, while in
the direction normal to the boundary the map moves points away from the
boundary. With the above extension, the vector field
$\Phi-\operatorname{id}$ has no winding around $p$. Hence
\[
\operatorname{ind}(p,\Phi)=0.
\]

Suppose next that $p$ corresponds to a stable direction. The
projectivized return map expands in the boundary direction, while the
map contracts in the direction normal to the boundary. In this case,
the vector field $\Phi-\operatorname{id}$ makes one negative turn around
$p$. Hence
\[
\operatorname{ind}(p,\Phi)=-1.
\]
This proves the lemma.
\end{proof}

\begin{proof}[\bf Proof of Theorem~\ref{thm:strongly}]
Consider the broken book decomposition described above. If it has a
broken binding component, the transverse homoclinic or heteroclinic
cycle described above and Lemma~\ref{lem:positive-hyp-from-horseshoe}
imply that $R_\lambda$ has infinitely many simple positive hyperbolic
periodic orbits.

Suppose now that the broken book decomposition has no broken binding
component. Let $S$ be a page of the resulting rational open book
decomposition and let
\[
\varphi:\operatorname{Int}S\to\operatorname{Int}S
\]
be the first return map. By Lemma~\ref{lem:birkhoff-boundary-index},
$\varphi$ extends to a homeomorphism of $S$ and satisfies the boundary
condition in Theorem~\ref{maintheoremonsurface}.

The map $\varphi$ preserves the area form
$
\omega=d\lambda|_{\operatorname{Int}S},
$
and
$
\int_{\operatorname{Int}S}\omega<\infty.
$
Moreover, $\varphi$ is nondegenerate, since the linearized return map
at a periodic point is conjugate to the linearized Poincar\'e return
map along the corresponding Reeb orbit.

It remains to show that $\varphi$ has infinitely many periodic points
in $\operatorname{Int}S$. By \cite[Theorem~1.2]{CDR}, a nondegenerate
Reeb vector field has either two or infinitely many simple periodic
orbits. Since $R_\lambda$ has at least three simple periodic orbits, it
has infinitely many simple periodic orbits. Only finitely many of them
are binding orbits. Hence infinitely many simple periodic orbits
intersect $\operatorname{Int}S$, and they give infinitely many periodic
points of $\varphi$.

Therefore, Theorem~\ref{maintheoremonsurface} applies to $\varphi$, and
$\varphi$ has infinitely many positive hyperbolic periodic points in
$\operatorname{Int}S$. Since each simple Reeb orbit intersects $S$ in
only finitely many points, these periodic points correspond to
infinitely many distinct simple positive hyperbolic periodic orbits of
$R_\lambda$. This completes the proof.
\end{proof}

\end{document}